\documentclass[a4paper,12pt]{scrartcl}

\usepackage[utf8]{inputenc}
\usepackage[ngerman,english]{babel}
\usepackage[T1]{fontenc}
\usepackage{amsmath}
\usepackage{amsthm}
\usepackage{amsfonts}
\usepackage{amssymb}
\usepackage{mathtools}
\usepackage{graphicx}
\usepackage{float}
\usepackage{enumerate}
\usepackage{listings}
\usepackage{pdfpages}
\usepackage{abstract}
\usepackage{titling}
\usepackage{csquotes}

\usepackage{tikz}
\usepackage{tikzscale}
\usepackage{tikz-qtree}
\usetikzlibrary{arrows, petri, topaths, calc, decorations.pathmorphing}
\usetikzlibrary{arrows.meta}
\usepackage{tkz-berge}
\usepackage[position=top]{subfig}

\usepackage[backend = biber, style = authoryear]{biblatex}
\usepackage{setspace}

\usepackage[paper=a4paper,left=30mm,right=30mm,top=25mm,bottom=30mm]{geometry} 

\usepackage{hyperref}

\deffootnote{1em}{1em}{\textsuperscript{\thefootnotemark\ }}

\newtheorem{lemma}{Lemma}
\newtheorem{corollary}[lemma]{Corollary}
\newtheorem{theorem}[lemma]{Theorem}

\theoremstyle{remark}

\newcommand{\N}{\mathbb{N}}

\newcommand{\ACAZERO}{\text{ACA}_0}
\newcommand{\ATRZERO}{\text{ATR}_0}
\newcommand{\PICAZERO}{\Pi_1^1\text{-CA}_0}
\newcommand{\PICA}{\Pi_1^1\text{-CA}}

\begin{document}

\onehalfspacing

\title{Upper bounds on the graph minor theorem}
\author{Martin Krombholz and Michael Rathjen}
\date{}

\maketitle

\begin{abstract}
Lower bounds on the proof-theoretic strength of the graph minor theorem were found over 30 years ago by \textcite{friedmanrobertsonseymour87}, but upper bounds have always been elusive. We present recently found upper bounds on the graph minor theorem and other theorems appearing in the Graph Minors series. Further, we give some ideas as to how the lower bounds on some of these theorems might be improved.

\end{abstract}

\section{Introduction}

Graph theory supplies many well-quasi-ordering theorems for proof theory to study. The best known of these is Kruskal's theorem, which as discovered independently by \textcite{schmidt79} and Friedman (published by \textcite{simpson85}) possesses an unusually high proof-theoretic strength that lies above that of $\ATRZERO$. This result was then extended by Friedman to extended Kruskal's theorem, a form of Kruskal's theorem that uses labelled trees for which the embedding has to obey a certain gap-condition, which was shown to have proof-theoretic strength just above even the theory of $\PICAZERO$, the strongest of the five main theories considered in the research program known as reverse mathematics. 

Reverse mathematics (RM) strives to classify the strength of particular theorems, or bodies of theorems, of ``ordinary'' mathematics by means of isolating the essential set existence principles used to prove them, mainly in the framework of subsystems of second order arithmetic. The program is often summarized by saying that there are just five systems, known as the ``Big Five'', that are sufficient for this classification. The picture of RM that we currently see, though, is more complicated: 

\begin{enumerate} 
\item Those parts of mathematics that have been analyzed in RM, are mostly results from the 19th century and the early 20th century with rather short proofs (varying from half a page to a few pages in length). By contrast, e.g., the large edifice of mathematics that Wiles' proof of {\em Fermat's Last Theorem} utilizes has not been analyzed in detail. 

\item By now there are quite a number of theorems that do not fit the mold of the Big Five. For instance, Ramsey's theorem for pairs, Kruskal's theorem and the graph minor theorem do not equate to any of them. For several others, such as Hindman's theorem, this is still an open question. 

\item There are areas of mathematics where complicated double, triple and more times nested transfinite inductions play a central role. Such proof strategies are particularly frequent in set theory (e.g. in fine structure theory and combinatorial theorems pertaining to $L$) and in higher proof theory (e.g. in the second predicative cut elimination theorem and the impredicative cut elimination and collapsing theorems). As RM is usually presented, one might be tempted to conclude that such transfinite proof modes are absent from or even alien to ``ordinary'' mathematics. However, they are used in the proof of the graph minor theorem. Are they really necessary for its proof? 

\end {enumerate} 

In this paper we will be concerned with the proof of the graph minor theorem, which is a fairly recent result. It has a very complicated and long proof that features intricate transfinite inductions. In particular, we will be analyzing these inductions and classify them according to principles that are familiar from proof theory and the foundations of mathematics. As to the importance attributed to the graph minor theorem, let's quote from a book on Graph Theory \textcite{diestel17}, p. 249. 
\begin{quote} 
\em Our goal $[\ldots]$ is a single theorem, one which dwarfs any other result in graph theory and may doubtless be counted among the deepest theorems that mathematics has to offer: {\em in every infinite set of graphs there are two such that one is a minor of the other}. This {\em graph minor theorem}, inconspicuous though it may look at first glance, has made a fundamental impact both outside graph theory and within. Its proof, due to Neil Robertson and Paul Seymour, takes well over 500 pages.
\end{quote} 

The starting point of this grand proof is the bounded graph minor theorem, i.e. the graph minor theorem restricted to those graphs of bounded ``tree-width''. The bounded graph minor theorem was connected to Friedman's extended Kruskal's theorem by \textcite{friedmanrobertsonseymour87}, and the two were even shown to be equivalent. This provided a natural example of a theorem of combinatorial mathematics that has extremely high proof-theoretic strength, and at the same time gave a lower bound on the graph minor theorem. While the precise proof-theoretic strength of the bounded graph minor theorem was established by \textcite{friedmanrobertsonseymour87}, the same was not the case for the full graph minor theorem, for which not even an upper bound was found, which no doubt was due to the fact that the proof's over 500 pages of complicated combinatorial arguments. In the following, we will thus outline how the graph minor theorem and other important theorems of the Graph Minors series, like the immersion theorem, can be proved in $\PICAZERO$ with the additional principles of $\Pi^1_3$-induction and $\Pi^1_2$-bar induction.

\section{Well-quasi-ordering theorems of the Graph Minors series}

The relations of minor and immersion can be understood as finding a certain expansion of one graph $G_1$ in another graph $G_2$. All graphs in this paper are finite and without loops unless noted otherwise, and we denote the vertex set of a graph $G$ by $V(G)$ and its edge set by $E(G)$. For the minor relation, define a minor-expansion of $G_1$ to be a function $f:G_1\longrightarrow G_2$ so that $v\in V(G_1)$ gets mapped to a connected subgraph $f(v)\subseteq G_2$ so that $f(v)\cap f(u)=\emptyset$ if $u\neq v$, and each edge $e\in E(G_1)$ gets mapped injectively to an edge $f(e)\in E(G_2)$ so that if the endpoints of $e$ are $u$ and $v$, then $f(e)$ connects vertices $u'\in f(u)$ and $v'\in f(v)$. If an expansion of $G_1$ is a subgraph of $G_2$, $G_1$ is said to be a minor of $G_2$, denoted $G_1\leq G_2$. An immersion relation between graphs $G_1$ and $G_2$ is similarly witnessed by an immersion-expansion $f:G_1\longrightarrow G_2$ so that vertices of $G_1$ are mapped injectively to vertices of $G_2$, and so that an edge $e$ with endpoints $u$ and $v$ is mapped to a path $f(e)$ in $G_2$ between $f(u)$ and $f(v)$ so that for distinct edges $e_1,e_2\in E(G_1)$ the paths $f(e_1)$ and $f(e_2)$ are edge-disjoint (but may intersect at vertices), i.e. $E(f(e_1))\cap E(f(e_2)) = \emptyset$. The graph minor and immersion theorem are then the following theorems.

\begin{theorem}[Graph minor theorem, \textcite{graphminorsxx}]
For every sequence $\left\langle G_i:i\in\N\right\rangle$ of graphs there are $i < j$ so that $G_i$ is a minor of $G_j$.

\end{theorem}

\begin{theorem}[Immersion theorem, \textcite{graphminorsxxiii}]
For every sequence $\left\langle G_i:i\in\N \right\rangle$ of graphs there are $i < j$ so that there is an immersion of $G_i$ into $G_j$.

\end{theorem}

The proof of the graph minor theorem can be divided into two major steps. First, the excluded minor theorem is proved, which takes up most of the Graph Minors series. The excluded minor theorem says that if one graph $G$ does not contain another graph $H$ as a minor, then $G$ has to have a certain structure, namely that it can be decomposed into parts which are connected in a tree-like shape and can almost be embedded into a surface into which $H$ can not be embedded. This is then used as follows: In a proof of the graph minor theorem, for any sequence of graphs $\left\langle G_1, G_2,\ldots\right\rangle$ one may assume that $G_1$ is not a minor of any $G_j$, $j > 1$, as otherwise the graph minor theorem holds. Thus, it suffices to prove the graph minor theorem for any sequence of graphs possessing the structure obtained by applying the excluded minor theorem for $G_1$, for any such $G_1$. This means that it is enough to prove the graph minor theorem for graphs which consist of parts connected in a tree-like shape that are almost embeddable into some fixed surface, which is the second major step of the proof of the graph minor theorem.

The proof of the excluded minor theorem is not very complex from a metamathematical point of view. This is due to the fact that surfaces are uniquely determined by their fundamental polygons, and that graph embeddings on any surface can thus be represented by a natural number encoding a graph drawing with rational coordinates in this fundamental polygon. With this approach, the entire proof of the excluded minor theorem does not feature any infinite objects nor any infinite proof techniques, and it is straightforward to carry it out in $\ACAZERO$, which will be our base theory in the following. The only papers of the Graph Minors series that use more advanced proof techniques are Graph Minors IV \cite*{graphminorsiv}, VIII \cite*{graphminorsviii}, XVIII \cite*{graphminorsxviii}, XIX \cite*{graphminorsxix}, XX \cite*{graphminorsxx} and XXIII \cite*{graphminorsxxiii}.

Graph Minors IV \cite*{graphminorsiv} proves in a sense an early version of the graph minor theorem for graphs with a certain structure as described above, namely the graph minor theorem for graphs that have bounded tree-width, a property which is defined in terms of tree-decompositions. A tree-decomposition of a graph $G$ is essentially a decomposition of $G$ into parts that are connected in a tree-like shape, i.e. a tree-decomposition of $G$ consists of a tree $T$ and for every $t\in V(T)$ a subgraph $G_t$ of $G$ so that 
\begin{itemize}
\item $\bigcup_{t\in V(T)}G_t = G$, and

\item if an edge $e$ of $T$ has endpoints $t_1$ and $t_2$, and $T_1$ and $T_2$ are the two components of $T$ obtained by removing $e$ from $T$, then every path in $G$ from some $v\in \bigcup_{t\in V(T_1)}G_t$ to some $u\in \bigcup_{t\in V(T_2)}G_t$ has to contain a vertex of $G_{t_1}\cap G_{t_2}$.

\end{itemize}
The width of such a tree-decomposition is then defined to be $\max_{t\in V(T)}\left|V(G_t)\right| - 1$. The tree-width $tw(G)$ of $G$ is the minimum width of all its tree-decompositions, and the bounded graph minor theorem can be stated as follows.

\begin{theorem}[Bounded graph minor theorem, \textcite{graphminorsiv}]
Let $n$ be a natural number, then in any sequence $\left\langle G_i:i\in\N\right\rangle$ of graphs so that $tw(G_i)\leq n$ for every $i\in\N$, there are $G_i$ and $G_j$ with $i < j$ so that $G_i$ is a minor of $G_j$. 
\end{theorem}

The bounded graph minor theorem has been analyzed from a metamathematical perspective by \textcite{friedmanrobertsonseymour87}, who determined that its proof-theoretic strength lies just above that of $\PICAZERO$. They observed that the bounded graph minor theorem can be proved for each individual tree-width in $\PICAZERO$, and since the bounded graph minor theorem is a $\Pi^1_1$-statement, that an application of $\Pi^1_1$-reflection for $\PICAZERO$ thus suffices to prove the bounded graph minor theorem. This approach circumvents a $\Pi^1_3$-induction, which is roughly used to show that some minimal bad sequence always exists under certain circumstances, and \textcite{friedmanrobertsonseymour87} in turn showed that no theory of lower proof-theoretic strength than $\PICAZERO$ augmented with $\Pi^1_1$-reflection for $\PICAZERO$ can prove the bounded graph minor theorem. There is however no such proof for some theorems of Graph Minors IV \cite*{graphminorsiv} which are more important for the rest of the Graph Minors series, and for these theorems only the upper bound of $\PICAZERO + \Pi^1_3\text{-IND}$ is known. \textcite{friedmanrobertsonseymour87} further showed that the bounded graph minor theorem is equivalent to the planar graph minor theorem, i.e. the graph minor theorem for those graphs which can be drawn (or equivalently, embedded) in the plane.

Graph Minors VIII \cite*{graphminorsviii} proves a generalization of the planar graph minor theorem. Define for every surface $\Sigma$ the $\Sigma$-graph minor theorem: 

\begin{theorem}[$\Sigma$-graph minor theorem]
For every sequence $\left\langle G_i:i\in\N \ldots\right\rangle$ of graphs that can be drawn in $\Sigma$ without crossings there are $i < j$ so that $G_i\leq G_j$.

\end{theorem} 

If $S^2$ denotes the sphere, then the planar graph minor theorem is just the $S^2$-graph minor theorem, since embeddability in the sphere and drawability in the plane are equivalent. Denote by $\forall\Sigma$-GMT the statement that the $\Sigma$-graph minor theorem holds for every surface $\Sigma$. It is shown in Graph Minors VIII that the $\Sigma$-graph minor theorem and $\forall\Sigma$-GMT are indeed true, and it can further be shown that both of these theorems are equivalent to the planar and hence also the bounded graph minor theorem. This is done by extending the proof that each instance of the bounded graph minor theorem is provable in $\PICAZERO$ all the way into Graph Minors VII \cite*{graphminorsviii}, so that it can be shown that for each surface $\Sigma$, the $\Sigma$-graph minor theorem is provable in $\PICAZERO$. An application of $\Pi^1_1$-reflection for $\PICAZERO$ then establishes the equivalence of $\forall\Sigma$-GMT and the planar graph minor theorem, and hence also that of $\forall\Sigma$-GMT and the bounded graph minor theorem. The results of \textcite{friedmanrobertsonseymour87} can thus be extended as follows, see \textcite{krombholz18}.

\begin{theorem}
The following are equivalent over $\ACAZERO$:
\begin{itemize}
\item The well-orderedness of the ordinal $\psi_0(\Omega_\omega)$,

\item Friedman's extended Kruskal's theorem,

\item the bounded graph minor theorem,

\item the planar graph minor theorem,

\item the $\Sigma$-graph minor theorem, for any surface $\Sigma$, and

\item $\forall\Sigma$-GMT.

\end{itemize}

\end{theorem}

The next use of strong infinitary proof-techniques is in Graph Minors XVIII \cite*{graphminorsxviii} which provides another restricted form of the graph minor theorem that facilitates the proof of the version of the graph minor theorem necessary for the second major step of the proof of the graph minor theorem outlined above. The theorem of Graph Minors XVIII \cite*{graphminorsxviii} in a sense allows one to focus on the individual pieces of the graph decomposition obtained by the excluded minor theorem, thereby avoiding the need to work with tree-decompositions. The theorem that these individual pieces of the above graph decomposition are well-quasi-ordered by the minor relation is then proved in Graph Minors XIX \cite*{graphminorsxix}. The proof of this version of the graph minor theorem requires a further very strong proof principle, namely that of $\Pi^1_2$-bar induction. In Graph Minors XX \cite*{graphminorsxx} these results are then combined to prove the full graph minor theorem. Finally, Graph Minors XXIII \cite*{graphminorsxxiii} proves the immersion theorem and a generalization of the graph minor theorem to hypergraphs in a certain sense. 

This generalization to hypergraphs can be stated as follows. For a vertex set $V$ denote by $K_V$ the complete graph on $V$, i.e. the graph with vertex set $V$ in which every two distinct vertices are connected by an edge. Then a collapse $f$ of $G_2$ to $G_1$ is a function mapping vertices of $G_1$ to disjoint connected subgraphs of $K_{V(G_2)}$ and edges of $G_1$ injectively to edges of $G_2$ so that $f(e)$ is incident with a vertex of $f(v)$ whenever $e$ is incident with $v$ for all $e\in E(G_1)$ and $v\in V(G_1)$, and further that for every vertex $v$ and every edge $e_v$ of $f(v)$ with endpoints $v_1$ and $v_2$, there must be an edge of $G_2$ that has among its endpoints the vertices $v_1$ and $v_2$. Further, if $Q$ is a well-quasi-order and the edges of $G_1$ and $G_2$ are labelled via functions $\phi_1:E(G_1)\longrightarrow Q$, $\phi_2:E(G_2)\longrightarrow Q$, then $f$ is also required to respect the edge labels of $G_1$ and $G_2$, in the sense that $\phi_1(e)\leq_Q\phi_2(f(e))$ has to hold for every edge $e\in E(G_1)$. Then Graph Minors XXIII \cite*{graphminorsxxiii} shows that the following generalization of the graph minor theorem holds.

\begin{theorem}
Let $Q$ be a well-quasi-order. Then in every infinite sequence $\left\langle G_i: i\in\N\right\rangle$ of $Q$-edge-labelled hypergraphs there are $j > i$ so that there is a collapse of $G_j$ to $G_i$ which respects the labels of $G_i$ and $G_j$.

\end{theorem}

Further, Graph Minors XXIII \cite*{graphminorsxxiii} also proves that similar labelled versions of the graph minor and immersion theorem hold. If $Q$ is a well-quasi-order and $\phi_1:E(G)\longrightarrow Q$, $\phi_2:E(G)\longrightarrow Q$ are labelling functions for the edges of $G_1$ and $G_2$, then a minor relation $G_1\leq G_2$ via an expansion $f$ is said to respect these labels if $\phi_1(e)\leq_Q \phi_2(f(e))$ for every edge $e\in G_1$. Similarly, for vertex-labelling functions $\phi_1:V(G)\longrightarrow Q$, $\phi_2:V(G)\longrightarrow Q$ the minor relation is said to respect the labels if for every $v\in V(G_1)$ there is a $v'\in f(v)$ so that $\phi_1(v)\leq_Q \phi_2(v')$. If $\phi_1$ and $\phi_2$ are vertex-labelling functions from a well-quasi-order $Q$ of $G_1$ and $G_2$ respectively, say that an immersion $f$ respects this labelling if $\phi_1(v)\leq_Q\phi_2(f(v))$ for every $v\in V(G)$. Then the labelled graph minor and immersion theorem are true as well.

\begin{theorem}[Labelled graph minor theorem]
Let $Q$ be a well-quasi-order and let $\left\langle G_i:i\in\N\right\rangle$ be a sequence of $Q$-vertex- and edge-labelled graphs. Then there are $i < j$ and a minor expansion $f:G_i\longrightarrow G_j$ that respects the labels of $G_i$ and $G_j$. 

\end{theorem}

\begin{theorem}[Labelled immersion theorem]
Let $Q$ be a well-quasi-order and let $\left\langle G_i:i\in\N\right\rangle$ be a sequence of $Q$-vertex-labelled graphs. Then there are $i < j$ and an immersion expansion $f:G_i\longrightarrow G_j$ that respects the labels of $G_i$ and $G_j$.

\end{theorem}

In order to prove these theorems, Graph Minors XXIII \cite*{graphminorsxxiii} requires another $\Pi^1_2$-bar induction similar to that used in Graph Minors XIX \cite*{graphminorsxix}. The bar induction of Graph Minors XIX \cite*{graphminorsxix} is used when assuming that a certain class of graph embeddings is minimal with respect to certain properties, in order to prove that the above mentioned sequence of graphs embedded in a surface is good. As said above, the graphs themselves might not actually be completely embeddable in the surface, and so the non-embeddable parts are coded as labels from a well-quasi-order, to provide a (now labelled) graph that is completely embeddable into the surface. When assuming that the set of possible labels is a minimal well-quasi-order so that the set of corresponding graphs is a counterexample, one essentially performs a $\Pi^1_2$-bar induction on a well-quasi-order.

\section{Bar induction in the Graph Minors series}

More precisely, in Graph Minors XIX \cite*{graphminorsxix} two $\Pi^1_2$-bar inductions and three ordinary $\Pi^1_2$-inductions need to be performed. These inductions take the form of the assumption that there is no minimal bad counterexample to a version of the graph minor theorem. This version of the graph minor theorem is for graphs that are embedded in a fixed surface and have labels from well-quasi-orders on the edges. Further, the minor relation between these graphs is altered in such a way that edges incident with a cuff stay fixed on the surface under minor-expansions, and so that it respects the labels of the well-quasi-order. The minimal counterexample to the graph minor theorem for such graphs is then required to have as few handles, crosscaps, cuffs and edges around cuffs as possible, which correspond to the ordinary $\Pi^1_2$-inductions mentioned above, since the well-quasi-orders for the edges are not required to be the same for ``smaller'' possible counterexamples. 

The $\Pi^1_2$-bar inductions then occur when requiring that the well-quasi-orders of the counterexample are also minimal with respect to the initial ideal ordering and so-called refinement relation. We present the bar induction corresponding to the initial ideal relation in greater detail to illustrate that it can deal with the induction principle actually performed in Graph Minors XIX \cite*{graphminorsxix}; the relation corresponding to refinement can be handled analogously. As already noted, the counterexample to our version of the graph minor theorem is required to have labels from a well-quasi-order that is minimal with regard to the initial ideal relation. A well-quasi-order $X$ is an initial ideal of another well-quasi-order $X'$, denoted $X\preceq X'$, if $X\subseteq X'$ and if $X$ is closed downward with regard to $X'$, that is if 
\[\forall x\in X\forall x'\in X'(x'\leq_{X'} x\rightarrow x'\in X).\]
Assuming that the counterexample has minimal well-quasi-orders with regard to this relation then corresponds to the induction scheme 
\[\forall X(WQO(X) \rightarrow (\forall X'\prec X (\forall X''\prec X' \varphi(X'') \rightarrow \varphi(X'))\rightarrow \varphi(X))).\]
This is different from the standard bar induction scheme, which postulates that 
\[\forall X(WF(X) \rightarrow \forall j(\forall i<_X j \varphi(i)\rightarrow \varphi(j))\rightarrow \forall n\in X \varphi(n)).\]
Further, it is not clear whether the induction scheme used in Graph Minors XIX \cite*{graphminorsxix} is actually implied by the usual bar-induction scheme, and it does not seem to be the case that this initial ideal induction scheme has been considered before in the literature of reverse mathematics. Note also that due to the different kinds of quantifiers present in second order arithmetic, it may for instance occur that the initial ideal induction scheme quantifies over uncountably many predecessor objects while the ordinary bar induction scheme is constrained to only countably many predecessor objects. Inspecting the proofs of Graph Minors XIX \cite*{graphminorsxix} further, it can however be discerned that a more restricted notion of initial ideal is sufficient to carry out the proofs. In the proofs of Graph Minors XIX \cite*{graphminorsxix}, the minimality of the counterexample with regard to this initial ideal relation is only used when a whole segment above a certain element is ``cut out'' of the well-quasi-ordering, that is only the relation $\preceq_1$ defined by
\[X'\prec_1 X :\Leftrightarrow \exists \left\langle x_1,\ldots,x_n\right\rangle\in X^{<\omega}\forall x'(x'\in X' \leftrightarrow x'\in X \land \forall i < n( x'\not\geq x_i))\]
is actually used in Graph Minors XIX \cite*{graphminorsxix}. Defining a relation $\leq_1$ (in other contexts known as the Smyth quasi-order) on the finite subsets $[X]^{<\omega}$ of a well-quasi-ordered set $X$ by 
\[\{y_1,\ldots,y_n\} \leq_1 \{z_1,\ldots,z_m\} :\Leftrightarrow \forall j\in\{1,\ldots,m\}\exists i\in\{1,\ldots,n\} y_i\leq z_j,\]
 and setting $X^{z_1,\ldots,z_n} := \{x\in X:\forall i < n(x\not\geq z_i)\}$ it can be shown that bar induction for $\leq_1$ implies initial ideal induction for $\preceq_1$:

\begin{lemma}\label{1bar}
Assume that for every well-quasi-ordered set $X^*$ and every $\Pi^1_2$-formula $\varphi'(n)$ the ordinary bar induction scheme holds with regard to $[X^*]^{<\omega}$ and $\leq_1$, i.e. that
\[\forall j(\forall i<_1 j \varphi'(i)\rightarrow \varphi'(j))\rightarrow \forall n\in [X^*]^{<\omega} \varphi'(n).\]
Then also the initial ideal induction scheme holds for every well-quasi-ordered set $X$ and every $\Pi^1_2$-formula $\varphi(Y)$ with regard to $\preceq_1$, i.e.
\[\forall X'\prec_1 X (\forall X''\prec_1 X' \varphi(X'') \rightarrow \varphi(X'))\rightarrow \varphi(X).\]

\end{lemma}

\begin{proof}
Note that if $X$ is well-quasi-ordered then $\leq_1$ is well-founded on $[X]^{<\omega}$ since a bad $\leq_1$-sequence in $X$ would in particular induce a bad $\preceq$-sequence in $X$ (see e.g. \textcite{forster03}), which is in contradiction to the well-quasi-orderedness of $X$. 

Now let $X$ be well-quasi-ordered and let $\top$ be a new element so that $\top > x$ for all $x\in X$. Define $\hat{X} := X\cup \{\top\}$. The idea for showing that the initial ideal induction scheme holds given the ordinary induction scheme is to encode the predecessors of $X$ with regard to $\preceq_1$ by finite subsets of $\hat{X}$, and to perform an ordinary bar induction on $[\hat{X}]^{<\omega}$ instead.

So assume that the usual bar induction scheme for $\Pi^1_2$-formulas with regard to $[\hat{X}]^{<\omega}$ and $\leq_1$ holds. Let $\varphi(X)$ be any $\Pi^1_2$-formula, then we need to show that $\prec_1$-initial ideal induction over $X$ holds for $\varphi$. Hence assume $\varphi$ is progressive with respect to $\prec_1$, i.e. that 
\[\forall X'\prec_1 X(\forall X''\prec_1 X'\varphi(X'')\rightarrow \varphi(X')).\] 
Then we need to show that $\varphi(X)$ holds. To do this, we define a formula $\varphi'(i)$ so that $\varphi'(\{y_1,\ldots,y_n\})$ essentially emulates $\varphi(\{x\in \hat{X}:\forall j < n: x\not\geq y_j\})$, as follows: 
\[\varphi'(i) := \forall Y(i = \{y_1,\ldots,y_n\} \rightarrow (\forall x(x\in Y\leftrightarrow x\in \hat{X}\land \forall j<n: x\not\geq_1 y_j) \rightarrow \varphi(Y))).\] 
By $\Sigma^0_0$-comprehension a set $Y$ satisfying the conditions in the antecedent always exists, and so $\varphi'$ is in fact the intended statement. Note that $\varphi'(i)$ is further still a $\Pi^1_2$-formula, and that we can thus utilize our idea to employ $\Pi^1_2$-bar induction for $\varphi'$ in order to show that $\varphi'(\{\top\})$ and hence $\varphi(X)$ holds. To this end we need to prove the progressiveness of $\varphi'$. So assume (letting $i$, $j$ be codes for finite subsets of $\hat{X}$) that $\forall i <_1 j\varphi'(i)$, then we need to show $\varphi'(j)$. 

For this, we first show that $\forall i <_1 j\varphi'(i)$ implies $\forall X''\prec_1 X^j \varphi(X'')$. But if $j = \{x_1,\ldots,x_m\}$, say, then $X''\prec_1 X^j$ means that $X'' = X^{x_1,\ldots,x_m,z_1,\ldots,z_k}$ for some $z_1,\ldots,z_k$, and trivially $\{x_1,\ldots,x_m,z_1,\ldots,z_k\} <_1 \{x_1,\ldots,x_m\}$, where the inequality must be strict since $X''\prec_1 X^j$. Let $i = \{x_1,\ldots,x_m,z_1,\ldots,z_k\}$. Then $\varphi'(i)$ holds since we assumed $\forall i <_1 j\varphi'(i)$, and since $X^i = X''$ we can infer that $\varphi(X'')$ holds as well. 

So we have shown that $\forall X''\prec_1 X^j \varphi(X'')$. Since $\varphi$ was assumed to be progressive with regard to $\prec_1$, this gives $\varphi(X^j)$ and therefore $\varphi'(j)$. This is what we needed to show for $\varphi'$ to be progressive. Since $\varphi'$ is progressive we can apply $\Pi^1_2$-bar induction on $\varphi'$ to obtain $\forall x\in [\hat{X}]^{<\omega} \varphi'(x)$. This gives us in particular $\varphi'(\{\top\})$, which in turn implies $\varphi(X)$ and thus completes the proof.

\end{proof}

In the above, finite sets of elements of $X$ are used to code the appropriate subsets of $X$. For the bar induction corresponding to the refinement relation, a finite sequence of such finite sets is needed instead. The critical condition of the refinement relation says in a sense that the well-quasi-orders from which some of edges are allowed to be labelled can be arranged in such a way that some of those well-quasi-orders are initial ideals of others, and at most identical. More precisely, a sequence $\left\langle X_1,\ldots, X_n\right\rangle$ is a refinement of a sequence $\left\langle X_1',\ldots,X_m'\right\rangle$ if $n\geq m$ and there is a function $f:\{1,\ldots,n\}\longrightarrow\{1,\ldots,m\}$ with the property that $X_i\preceq X_{f(i)}$ for all $i\leq n$, so that additionally $X_i,X_j\prec X_{f(i)}$ whenever $f(i) = f(j)$ for $i\neq j$, and so that $X_i\prec X_{f(i)}$ for some $i$. As in the previous induction, the $\prec$-relations are not actually required in their full form and can be replaced by $\prec_1$ relations, which enables us to perform a bar-induction in order to simulate the induction corresponding to the refinement relation. We write $\left\langle X_1,\ldots, X_n\right\rangle\prec_2\left\langle X_1',\ldots,X_m'\right\rangle$ if $\left\langle X_1,\ldots, X_n\right\rangle$ is a refinement of $\left\langle X_1',\ldots,X_m'\right\rangle$. To perform the bar-induction, we need a relation corresponding to $\prec_2$. As above, denote the set of finite subsets of a set $Y$ by $[Y]^{<\omega}$, and use $\rho$ and $\sigma$ as variables for such finite subsets. Define then on $([X]^{<\omega})^{<\omega}$ a relation $<_2$ by
\begin{align*}
\left\langle \rho_1,\ldots,\rho_n\right\rangle <_2 \left\langle \sigma_1,\ldots,\sigma_m\right\rangle  
:\Leftrightarrow \ & \exists f:\{1,\ldots,n\}\longrightarrow\{1,\ldots,m\}( \\
& \quad\forall i\leq n(\rho_i \leq_1 \sigma_{f(i)})\land \exists i\leq n(\rho_i <_1 \sigma_{f(i)})\land \\
& \quad\forall i,j(i\neq j\land f(i)=f(j)\rightarrow \rho_i <_1 \sigma_{f(i)}) ).
\end{align*}

In order to be able to carry out a bar-induction along this relation, we need to show that it is well-founded. This is done in the next lemma.

\begin{lemma}
Let $X$ be a well-quasi-ordered set. Then $([X]^{<\omega})^{<\omega}$ is well-founded with regard to $\leq_2$.

\end{lemma}
\begin{proof}
Because $X$ is well-quasi-ordered, $[X]^{<\omega}$ is well-founded with regard to $<_1$ by the remarks in the proof of the above lemma. Our aim is to employ K\"onig's lemma in order to show that there can be no infinite descending $\leq_2$-sequence in $([X]^{<\omega})^{<\omega}$. Thus if $\left\langle \rho_1,\ldots,\rho_n\right\rangle <_2 \left\langle \sigma_1,\ldots,\sigma_m\right\rangle$ via $f$, we say that $\sigma_j$ branches into $\rho_{i_1},\ldots,\rho_{i_{m_j}}$ if $f^{-1}(j) = \{i_1,\ldots,i_{m_j}\}$ and $\rho_{i_1} <_1 \sigma_j$ (which is immediate if $f^{-1}(j)$ consists of more than one element). 

Now assume that there is a sequence $s:=\left\langle \left\langle \rho_1^i,\ldots,\rho^i_{n_i}\right\rangle:i\in\N\right\rangle$ so that $s(i) >_2 s(i + 1)$ for all $i$, and let $\left\langle f_i:\{1,\ldots,n_i\}\longrightarrow \{1,\ldots,n_{i - 1}\}\right\rangle_{i \geq 2}$ be the corresponding sequence of functions witnessing the $<_2$ relations. In order to avoid confusing duplicate elements that may appear multiple times in that sequence, we interpret each $\rho^i_k$ as a term, and identify two such terms transitively if $\rho^{i + 1}_k = \rho^i_l$ and $f_{i + 1}(k) = l$. 

We now turn toward defining the tree we want to use K\"onig's lemma on. Let $S=\{\rho^i_k:i\in\N\land k\leq n_i\}$, and for $\rho,\sigma\in S$ define $\sigma$ to be a successor of $\rho$ if at some step in $s$ an element underlying $\rho$ branches into an element underlying $\sigma$. Note that due to the definition of $<_2$ every $\rho$ can branch only once, and that it can only branch into finitely many successors. This successor relation thus defines a forest on $S$, which is infinite since $s$ is an infinite descending sequence and in which every tree is finitely branching. Since this forest consists of $n_1$ and hence finitely many trees, one of these trees must be infinite as well. We can thus apply K\"onig's Lemma to this tree to obtain an infinite, strictly decreasing $<_1$-sequence in $[X]^{<\omega}$, which is a contradiction since $[X]^{<\omega}$ is well-founded by $<_1$.

\end{proof}

Similarly to $\prec_1$-initial ideal induction, we can now prove a lemma that shows that ordinary bar induction for $\leq_2$ implies the induction scheme corresponding to refinement. This is made precise in the following lemma.

\begin{lemma}
Assume that for every well-quasi-ordered set $X^*$ and every $\Pi^1_2$-formula $\varphi'(n)$ the bar induction scheme holds with regard to $([X^*]^{<\omega})^{<\omega}$ and $\leq_2$, i.e. that 
\[\forall j(\forall i<_2 j \varphi'(i)\rightarrow \varphi'(j))\rightarrow \forall n\in ([X^*]^{<\omega})^{<\omega} \varphi'(n).\]
Then for every finite sequence of well-quasi-ordered sets $X := \left\langle X_1,\ldots, X_n\right\rangle$ and every $\Pi^1_2$-formula $\varphi(Y)$ the induction scheme corresponding to refinement 
\[(\forall X'\prec_2 X (\forall X''\prec_2 X' \varphi(X'') \rightarrow \varphi(X'))\rightarrow \varphi(X))\]
holds as well.

\end{lemma}
\begin{proof}
The proof is essentially the same as the one for Lemma \ref{1bar}.

\end{proof}

This shows that the critical parts of Graph Minors XIX \cite*{graphminorsxix} can be dealt with by a $\Pi^1_2$-bar induction. A similar induction is performed in the proof of the immersion theorem in Graph Minors XXIII \cite*{graphminorsxxiii} that can be dealt with by the same techniques. To give an overview, based on unpublished research we have the following placements of proof-theoretic strength:

\begin{enumerate}[a)]
\item $|\Pi^1_1{-}{\mathbf{CA}}_0| =\psi_0(\Omega_{\omega})$. 

\item $|\Pi^1_1{-}{\mathbf{CA}}_0+\Pi^1_2\mbox{-IND}| =\psi_0(\Omega_{\omega}{\cdot}\omega^{\omega})$. 

\item$|\Pi^1_1{-}{\mathbf{CA}}| =\psi_0(\Omega_{\omega}{\cdot}\varepsilon_0)$. 

\item $|\Pi^1_1{-}{\mathbf{CA}}_0+\Pi^1_2\mbox{-BI}| = \psi_0(\Omega_{\omega}^{\omega})$. 

\item $|\Pi^1_1{-}{\mathbf{CA}}_0+\Pi^1_2\mbox{-BI}+\Pi^1_3\mbox{-IND}| = \psi_0(\Omega_{\omega}^{\omega^{\omega}})$. 

\item $\psi_0(\Omega_{\omega})\;<\;\mbox{ordinal of graph minor and immersion theorems}\;\leq \; \psi_0(\Omega_{\omega}^{\omega^{\omega}})$.

\end{enumerate}

\section{Possible lower bound improvements}

To narrow down the corridor in which the proof-theoretic strength of the theorems considered above lies, one might try to increase their lower bounds. The immersion theorem with well-quasi-ordered labels seems to be particularly suited for such a task, since it almost imposes an approach similar to that of Friedman's extended Kruskal's theorem $EKT$\cite*{simpson85}. There, a function is used to relate labelled trees ordered by embedding with gap-condition to ordinals from the ordinal notation system $OT(\Omega_\omega)$. This ordinal notation system is used for the ordinal analysis of $\PICAZERO$, which shows that $\left|\PICAZERO\right| = \Psi_0(\Omega_\omega)$, and derived from the set $C_0(\Omega_\omega)$ from \textcite{buchholz86}. In \textcite{simpson85} it is then shown that the above approach yields:

\begin{theorem}
$\ACAZERO\vdash EKT\rightarrow WO(\Psi_0(\Omega_\omega))$. In particular, $EKT$ is not provable in $\PICAZERO$.

\end{theorem}

Similar to $EKT$, a principle $GKT_\omega(Q)$, denoting generalized Kruskal's theorem with labels from $\omega$ and additional well-quasi-ordered labels from a well-quasi-order $Q$, can be defined as follows. First, the objects related to this principle are rooted trees $T$ that have two labelling functions associated with them, one function $l:V(T)\longrightarrow \omega$ and another function $l_Q:V(T)\longrightarrow Q$. They are ordered by embeddings $f:T_1\longrightarrow T_2$ that satisfy the gap-condition 
\[\forall x\in V(T_1)\forall y\in V(T_2)(y\leq f(x)\land \lnot\exists z\in V(T_1)(z< x\land y\leq f(z))\rightarrow l(y)\geq l(x)),\]
and additionally respect the labels from $Q$ in the sense that
\[\forall x\in V(T_1)(l_Q(x)\leq l_Q(f(x))).\]
For any vertex $v\neq root(T)$ in such a tree, if $w$ is the first vertex on the path from $v$ to $root(T)$, we define $T^v$ to be the component of $T\setminus w$ which includes $v$, and set $root(T^v) := v$. Then one can relate ordinals to a subset of these trees, by decreeing that the well-quasi-order $Q$ have the form $Q = W_Q\cup\{+,\omega^\cdot,\psi\}$, where $W_Q$ is a well-order and the elements of $\{+,\omega^\cdot,\psi\}$ are incomparable to all others, in the following way. First, we need an ordinal notation system $OT(\Omega_\omega \cdot W)$ from \textcite{rathjenthompson} which relativizes $OT(\Omega_\omega)$ by putting $\sup(W)$ many copies of $\Omega_\omega$ above $\Omega_\omega$. Interpret a well-order $W$ as an ordinal and for $w\in W$ set $\overline{w}:= \Omega_\omega \cdot (1 + w)$. Define then sets  $C^W_m(\alpha)$, $m\in\N$, and collapsing functions $\psi^W_m (\alpha)$, $m\in\N$ by induction on $\alpha$. Let $C^W_m(\alpha)$ be the least set $C\supseteq \Omega_m\cup\{\Omega_i:i\in \N\}\cup \{\overline{w}:w\in W\}$ so that: 
\begin{itemize}
\item $C\cap\Omega_\omega$ is closed under $+$ and $\omega^\cdot$, 

\item $\overline{w} + \alpha\in C$ whenever $w\in W$ and $\alpha \in C\cap\Omega_\omega$, and

\item $C\cap\alpha$ is closed under $\psi_n$ for all $n\in\N$.

\end{itemize}

Then we can define $\psi^W_m (\alpha)$ by
\[\psi^W_m (\alpha) := \min\{\xi : \xi\notin C^W_m(\alpha)\}.\]

We also write $\psi_m$ instead of $\psi^W_m$ if no confusion is possible. The proof-theoretic ordinal of $\PICA$ in terms of these collapsing functions is then $\psi_0(\Omega_\omega \cdot \varepsilon_0)$. Let $w' := \sup(W)$. In the following we will always assume that ordinals are in normal form with regard to the ordinal notation system $OT(\Omega_\omega\cdot\ W)$ that corresponds to $C_0\left(\overline{w'}\right)$; see \textcite{rathjenthompson} for details.

To define the ordinal related to a tree, we additionally assume that $W$ has a special element $w_0$ so that $w_0 < w$ for all $w\in W\setminus\{w_0\}$ (normally $w_0$ would correspond to $0$, but we need it to be ``less than'' $0$). We then define $\psi_m(w_0) := \Omega_m$, and to simplify notation, we define further $\psi_m(w + \alpha) := \psi_m(\overline{w} + \alpha)$ for all $w\in W\setminus\{w_0\}$. A tree $T$ can then be assigned an ordinal $o(T)$ from $OT(\Omega_\omega \cdot W)\cap \Omega_\omega$ as follows:

\begin{itemize}
\item If $l_Q(root(T))\in W$ and $root(T)$ has no successor, then set $o(T) := \psi_n (w)$, where $n = l(root(T))$ and $w = l_Q(root(T))$.

\item If $l_Q(root(T))\in W\setminus\{w_0\}$ and $root(T)$ has one successor $v$, then set $o(T) := \psi_n (w + o(T^v))$, where $n = l(root(T))$ and $w = l_Q(root(T))$.

\item If $l_Q(root(T)) = +$ and $v_1$, $v_2$ are the successors of $root(T)$ ordered so that $o(T^{v_1}) \geq o(T^{v_2})$, then set $o(T) := o(T^{v_1}) + o(T^{v_2})$.

\item If $l_Q(root(T)) = \omega^\cdot$ and $v$ is the successor of $root(T)$, then set $o(T) := \omega^{o(T^{v})}$.

\item If $l_Q(root(T)) = \psi$ and $v$ is the successor of $root(T)$, then set $o(T) := \psi_n o(T^{v})$, where $n = l(root(T))$.

\item If none of these cases can be applied, $T$ is not assigned an ordinal.

\end{itemize}

In the following we will restrict ourselves to trees that can be assigned an ordinal as above, and well-quasi-orders suitable for labelling those trees. Then it can be shown that:

\begin{theorem}[\textcite{krombholz18}]
Let $Q$ be a well-quasi-order and $T_1$, $T_2$ be trees as above. Then $o(T_1)\leq o(T_2)$ whenever $T_1\leq T_2$. 

In particular, $GKT_\omega(Q)$ implies the well-orderedness of $OT(\Omega_\omega \cdot W_Q)$. 

\end{theorem}

From which, letting $GKT_\omega(\forall Q) := \forall Q( WQO(Q)\rightarrow GKT_\omega(Q))$, follows immediately:

\begin{theorem}
$\ACAZERO\vdash GKT_\omega(\forall Q)\rightarrow [\forall X(WO(X)\rightarrow WO(OT(\Omega_\omega\cdot X)))]$.

\end{theorem}

Then, observing that $\left|\PICA\right| = \Psi_0(\Omega_\omega\cdot \varepsilon_0)$,  we get stronger lower bounds on $GKT_\omega(\forall Q)$ (and in fact even $GKT_\omega(\varepsilon_0)$).

\begin{corollary}
$\PICAZERO + GKT_\omega(\forall Q)$ proves $WO(\psi_0(\Omega_\omega\cdot\varepsilon_0))$.

\end{corollary} 

\begin{corollary}
$\PICA\not\vdash GKT_\omega(\forall Q)$.

\end{corollary}

This idea might possibly be leveraged in the following way, by extending it to theorems of the Graph Minors series. Recall that an immersion of one graph $G_1$ into another graph $G_2$ is an injective function $f:G_1\longrightarrow G_2$ that maps vertices injectively to vertices and edges to edge-disjoint paths (the paths may intersect at vertices however). Given a labelled tree $T$ as in the statement $GKT_\omega(Q)$ with $Q = W_Q\cup \{+,\omega^\cdot,\psi\}$, one can then define a tree-like graph which under immersion expansion aims to behave like the labelled tree. 

Set $Q' := Q\cup\{root\}$ where $root$ is incomparable to all other elements of $Q'$, and define $V(G) := V(T)\cup \{r\}$, where $r$ is a new vertex. Set further $l_{Q'}(v) := l_Q(v)$ if $v\in V(T)$ and set $l_{Q'}(r) := root$. Connect then vertices $v$ of $G$ to their immediate predecessors by $l(v) + 1$ parallel edges, and connect $root(T)$ to $r$ by $l(root(T)) + 1$ parallel edges. We then adopt the notation $v\leq u$ if when deleting edges in $G$ until no multiple edges remain (which results in a tree), $v$ lies on the unique path from $u$ to the vertex labelled with $root$ in $G$. We also speak of predecessors and successors in $G$ with regard to this ordering. For $v$ in $V(G)$ define then $G^v$ to be the induced subgraph of $G$ with vertex-set $\{u\in V(G): v\leq u\}\cup \{r'\}$ where $r'$ is a new vertex labelled with $root$, and where $r'$ is connected to $v$ by as many edges as $v$ was connected to its immediate predecessor $p(v)$ in $G$. For vertices $v$ not labelled with $root$ set further $l(v) := \left|\{e\in E(G): \text{ $e$ connects $v$ and $p(v)$}\}\right| - 1$ (which is the same as $l(v)$ in $T$).

One can then relate an ordinal to $G$ in the obvious way, by definining $o(G)$ as follows:
\begin{itemize}
\item If the successor $v$ of $r$ is labelled from $W$ and $v$ has no successors, let $o(G) := \psi_{l(v)}(l_{Q'}(v))$.

\item If the successor $v$ of $r$ is labelled from $W$ and $v$ has a successor $w$, let $o(G) := \psi_{l(v)}(l_{Q'}(v) + o(G^w))$.

\item If the successor $v$ of $r$ is labelled with $+$, set $o(G) := o(G^{w_1}) + o(G^{w_2})$, where $w_1$ and $w_2$ are the successors of $v$ so that $o(G^{w_1})\geq o(G^{w_2})$.

\item If the successor $v$ of $r$ is labelled with $\omega^\cdot$, set $o(G) := \omega^{o(G^{w})}$, where $w$ is the successor of $v$.

\item If the successor $v$ of $r$ is labelled with $\psi$, set $o(G) := \psi_{l(v)}o(G^{w})$, where $w$ is the successor of $v$.

\end{itemize}

One could hope that $o(G_1)\leq o(G_2)$ whenever $G_1$ can be immersed into $G_2$, but sadly this result has not been established yet. When doing the proof for labelled trees, an induction on the height of the tree with additional induction hypotheses is usually used. However, aside from mapping the vertex labelled with $root$ in $G_1$ to the vertex labelled with $root$ in $G_2$, an immersion from $G_1$ into $G_2$ does not have to respect the ``tree-structure'' of $G_1$, as illustrated in figure \ref{immersion}. 

\begin{figure}[ht]
\begin{center}
\begin{tikzpicture}[transform shape, scale = 1.2]
\begin{scope}[every node/.style={circle, thick, draw, inner sep = 1 pt, minimum size = 0.5cm}]
    \node (R) at (4.5,-0.5) {$r$};
    \node (B) at (4,0.5) {$+$};
    \node (C) at (3.3,1.5) {$\varepsilon_0$};
    \node (D) at (4.7,1.5) {$\omega$};
		
		\node (R*) at (9,-1) {$r$};
		\node (A*) at (8.5,0) {$+$};
		\node (E*) at (9.2,1) {$\omega^2$};
    \node (B*) at (8,1) {$+$};
    \node (C*) at (7.3,2) {$\varepsilon_0$};
    \node (D*) at (8.7,2) {$0$};
\end{scope}

\begin{scope}[every node/.style={inner sep = 0.5 pt, minimum size = 0cm}]
\begin{scope}[every edge/.style={draw=black,very thick}]
		
		\path [-] (R*) edge (A*);
		\path [-, bend right] (R*) edge (A*);
		\path [-, bend left] (R*) edge (A*);
		\path [-] (A*) edge (B*);
		\path [-, bend right] (A*) edge (B*);
		\path [-, bend left] (A*) edge (B*);
		\path [-] (A*) edge (E*);
		\path [-, bend right] (B*) edge (C*);
		\path [-] (B*) edge (C*);
		\path [-, bend left] (B*) edge (C*);
		\path [-, bend right] (B*) edge (D*);
		\path [-, bend left] (B*) edge (D*);
		
		\path [-, bend right] (R) edge (B);
		\path [-, bend left] (R) edge (B);
		\path [-, bend right] (B) edge (C);
		\path [-, bend left] (B) edge (C);
		\path [-] (B) edge (D);
		
		\path [->, >=stealth, dashed] (R) edge (R*);
		\path [->, >=stealth, dashed] (B) edge (B*);
		\path [->, >=stealth, dashed, bend left] (C) edge (C*);
		\path [->, >=stealth, dashed, out = 45, in = 60] (D) edge (E*);

\end{scope}
\end{scope}

\end{tikzpicture}
\end{center}
\caption{One example where a valid immersion embedding does not respect ``infima'' of the graphs. The labels of the vertices are drawn inside the nodes, with $r$ used instead of $root$. The vertex map of the immersion embedding is given by the dashed arrows, with the edge map implied in the obvious way.}
\label{immersion}
\end{figure}
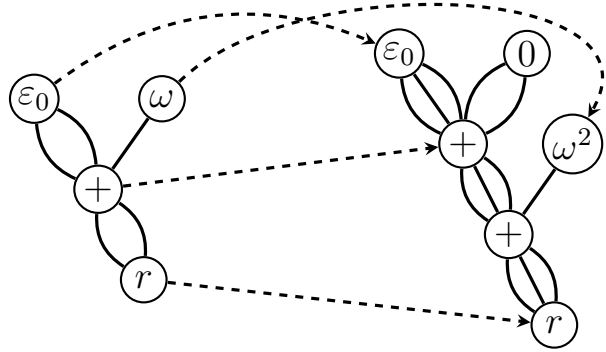

The induction hypotheses necessary for proving $o(G_1)\leq o(G_2)$ can not always be used in such a case, which makes the proof that this holds (if it should indeed hold) a lot harder. It should be noted that the immersion relation between two such graphs corresponds to a root preserving embedding $f$ between edge-labelled trees that is not order or infimum preserving (i.e. so that $f$ maps vertices injectively to vertices and edges to paths that do not have to be disjoint), that however satisfies a different gap-condition, namely that for $e\in E(G_2)$ it has to hold that $l(e)\geq \sum_{e'\in f^{-1}(e)}{l(e')}$, where $f^{-1}(e)$ denotes the set of edges $e'$ so that $e$ is an edge of $f(e')$. 

While it is not clear whether this construction works with immersions due to the above, it should be noted that it does work when using directed graphs and immersions, i.e. so that edges are directed from $u$ to $v$ if $u\leq v$ and so that an immersion expansion maps edges to edge-disjoint directed paths. However, the immersion theorem is known to not hold for the class of all directed graphs in general, and it is currently an ongoing effort in graph theory to establish for which classes of directed graphs it does hold. Thus, it is an open question whether lower bounds like these can be established for a more natural class of directed graphs, and further whether these results can be extended to undirected immersions.

\section*{Acknowledgments} 

The first author was supported by a scholarship from the University of Leeds
(``University of Leeds 110th Anniversary Scholarship'').\footnote{The opinions expressed in this publication are those of the authors and do not necessarily reflect the views of the University of Leeds.}

The second author was supported by a grant from the John Templeton Foundation
(``A new dawn of intuitionism: mathematical and philosophical advances'', ID 60842).\footnote{The opinions expressed in this publication are those of the authors and do not necessarily reflect the views of the John Templeton Foundation.}

\printbibliography

\end{document}